\documentclass[11pt, letterpaper]{elsarticle}
\usepackage{amsmath}
\usepackage{amssymb}
\usepackage{amsfonts}
\usepackage{graphicx}
\usepackage{multirow}
\usepackage{latexsym}
\usepackage{epic}
\usepackage{url}
\usepackage{soul}
\usepackage{afterpage}
\usepackage[unicode, pdftex]{hyperref}
\usepackage{multicol}
\usepackage{bbm}
\usepackage{setspace}
\usepackage{enumitem}
\usepackage{cleveref}
\usepackage[toc]{appendix}
\usepackage[norelsize,ruled, lined,linesnumbered]{algorithm2e}
\usepackage{amsthm}
\usepackage[symbol]{footmisc}
\usepackage[mathscr]{euscript}
\usepackage[left=0.9in,top=0.9in,right=0.9in,bottom=1in,nohead]{geometry}
\usepackage{empheq}
\usepackage{mathtools}
\DeclareMathOperator{\argmax}{argmax}
\DeclareMathOperator{\argmin}{argmin}

\makeatletter
\newtheorem*{rep@theorem}{\rep@title}
\newcommand{\newreptheorem}[2]{%
	\newenvironment{rep#1}[1]{%
		\def\rep@title{#2 \ref{##1}}%
		\begin{rep@theorem}}%
		{\end{rep@theorem}}}
\makeatother
\SetAlgoNlRelativeSize{-1}
\SetKwFor{For}{for}{}{end}
\SetKwFor{While}{while}{}{end}
\newreptheorem{example}{Example}
\newtheorem{theorem}{Theorem}

\newtheorem{lemma}{Lemma}

\newtheorem{corollary}{Corollary}
\newtheorem{remark}{Remark}

\usepackage{xcolor}
\definecolor{forestgreen}{RGB}{34,139,34}
\usepackage[colorinlistoftodos,prependcaption,textsize=tiny]{todonotes}
\usepackage{xargs}
\usepackage{booktabs}
\usepackage{caption}
\usepackage{subcaption}

\usepackage{lineno}
\usepackage{float}
\usepackage[section]{placeins}
\usepackage{arydshln}
\usepackage{latexsym}
\usepackage{epic}
\usepackage{amsmath,amsthm,amssymb}
\usepackage{graphicx}
\usepackage{url}
\usepackage{pgfplots}\usetikzlibrary{plotmarks}
\usepackage{soul}
\usepackage{afterpage}
\usepackage{cleveref}
\usepackage{bbm}
\usetikzlibrary{decorations.markings}
\usepackage{setspace}
\usepackage{appendix}
\usepackage{comment}
\DeclareGraphicsExtensions{.pdf,.png,.jpg,.bmp}

\allowdisplaybreaks

\makeatletter
	\def\ps@pprintTitle{%
		\let\@oddhead\@empty
		\let\@evenhead\@empty
		\let\@oddfoot\@empty
		\let\@evenfoot\@empty
	}
	\makeatother

\begin{document}
	\begin{frontmatter}
\title{\onehalfspacing On Big-$M$ Reformulations of Bilevel Linear Programs: \\ Hardness of A Posteriori Verification}
		
		\author[label1]{Sergey S.~Ketkov\footnote[2]{Corresponding author. Email: sergei.ketkov@business.uzh.ch; phone: +41 78 301 85 21.}}
		\author[label1]{Oleg A. Prokopyev}
		\address[label1]{Department of Business Administration, University of Zurich, Zurich, 8032, Switzerland}
		\begin{abstract}
			\onehalfspacing	
			\nolinenumbers		
	 A standard approach to solving optimistic \textit{bilevel linear programs} (BLPs) is to replace the lower-level problem with its Karush--Kuhn--Tucker (KKT) optimality conditions and reformulate the resulting complementarity constraints using auxiliary binary variables. This yields a single-level mixed-integer linear programming (MILP) model involving \textit{big}-$M$ parameters. 
	 While sufficiently large big-$M$\hspace{-0.5pt}s  that preserve at least one bilevel-optimal solution \color{black} can be computed in polynomial time, verifying \textit{a priori} that given big-$M$\hspace{-0.5pt}s do not cut off any~feasible or optimal lower-level solutions is known to be computationally difficult. In this paper, we establish two complementary hardness results. 
	 First, we show that, even with a single potentially incorrect big-$M$ parameter, it is $\mathrm{coNP}$-complete to verify \textit{a posteriori} whether  a given \color{black} optimal solution of the resulting~MILP model is bilevel
	 optimal. 
	 In particular, this negative result persists for min-max problems without coupling constraints and applies to strong-duality-based reformulations of mixed-integer~BLPs. Second,~we show that verifying global big-$M$ correctness remains computationally difficult \textit{a posteriori}, even when an optimal solution of the MILP model is available. 
		\end{abstract}
		
	\end{frontmatter}
	\onehalfspacing
\section{Introduction} \label{sec: intro}
\textit{Bilevel optimization} models a hierarchical interaction between two decision-makers, commonly referred to as a \textit{leader} and a \textit{follower}, each with  their \color{black} own objective function and constraints. The leader makes a decision first, optimizing  their \color{black} objective function and anticipating that the follower subsequently solves  their \color{black} own optimization problem, parameterized by the leader’s decision. Comprehensive surveys of bilevel optimization problems and their applications can be found in \cite{Colson2007, Kleinert2021, Sinha2017}. 

In this study, we analyze a canonical class of bilevel optimization problems in which both the leader and the follower solve linear programs (LPs). Specifically, we consider  optimistic \color{black} \textit{bilevel linear programs} (BLPs) of the~form:
\begin{subequations}
	\label{BLP}
	\begin{align}
		[\textbf{BLP}]: \quad
		\min_{\mathbf{x}, \mathbf{y}^*} \; &
		\Big\{ \mathbf{a}^{\top}\mathbf{x} + 
		\mathbf{d}^{\top}\mathbf{y}^* \Big\}
	  \label{obj: leader} \\
		\text{s.t. } 
		& \mathbf{H} \mathbf{x} + 
		\mathbf{G} \mathbf{y}^* \leq \mathbf{h} \label{cons: leader} \\
		& \mathbf{y}^* \in 
		\argmin_{\;\mathbf{y}} \,
		\Big\{\mathbf{g}^{\top}\mathbf{y}: \; 	\mathbf{L} \mathbf{x} + 
		\mathbf{F} \mathbf{y} 
		\leq
		\mathbf{f} \Big\}, \label{cons: follower}
	\end{align}
\end{subequations}
where $\mathbf{H} \in \mathbb{Q}^{m_l \times n_l}$,
$\mathbf{G} \in \mathbb{Q}^{m_l \times n_f}$, 
$\mathbf{L} \in \mathbb{Q}^{m_f \times n_l}$, 
$\mathbf{F} \in \mathbb{Q}^{m_f \times n_f}$, $\mathbf{h} \in \mathbb{Q}^{m_l}$, $\mathbf{f} \in \mathbb{Q}^{m_f}$, $\mathbf{a} \in \mathbb{Q}^{n_l}$, $\mathbf{d} \in \mathbb{Q}^{n_f}$ and~$\mathbf{g} \in \mathbb{Q}^{n_f}$; see, e.g., \cite{Audet1997, Hansen1992}. Accordingly, $n_i$ and $m_i$, $i \in \{l, f\}$, correspond to the numbers of variables and constraints for the leader and the follower, respectively. 

By construction, the bilevel problem [\textbf{BLP}] is formulated in the \textit{optimistic} sense.
That is, whenever the follower’s problem in (\ref{cons: follower}) has multiple optimal solutions, the leader selects one that also minimizes  their \color{black} objective function (\ref{obj: leader}) and satisfies the \textit{coupling constraints}~(\ref{cons: leader}).  In this regard, we note that the majority of the bilevel optimization literature focuses on the optimistic setting; see, e.g., \cite{Audet1997, Buchheim2023, Kleinert2020} and the survey in \cite{Kleinert2021}. Moreover, pessimistic BLPs with coupling constraints can, under suitable assumptions, be reformulated as optimistic BLPs without coupling constraints~\cite{Henke2025a, Henke2025b}.  

From a computational complexity perspective, the optimistic problem~[\textbf{BLP}] is already challenging: it is $\mathrm{NP}$-hard in general~\cite{BenAyed1990}, its min-max version is strongly $\mathrm{NP}$-hard~\cite{Hansen1992}, and even verifying local optimality for the leader remains $\mathrm{NP}$-hard~\cite{Vicente1994}. \color{black}
 In particular, \color{black} we say that~[\textbf{BLP}] has no coupling constraints when $\mathbf{G} = \mathbf{0}$, and we refer to it as a \textit{min-max} problem when~$\mathbf{G} = \mathbf{0}$ and~$\mathbf{g} = -\mathbf{d}$.  
Notably, in the min-max case, the optimistic and the pessimistic formulations of~[\textbf{BLP}] coincide. That is, all complexity results and solution approaches established in this paper for min-max BLPs apply equally to the corresponding pessimistic formulation. \color{black}




A standard approach to solving the optimistic problem [\textbf{BLP}] is to replace the follower’s problem in (\ref{cons: follower}) with its Karush-Kuhn-Tucker (KKT) optimality conditions; see, e.g., \cite{Audet1997, Fortuny1981}. That is,~[\textbf{BLP}] can be expressed as a mathematical program with equilibrium constraints (MPEC) given by:
\begin{subequations}
	\label{MPEC}
	\begin{align}
		\min_{\mathbf{x}, \mathbf{y}, \boldsymbol{\lambda}} \; & 
			\Big\{ \mathbf{a}^{\top}\mathbf{x} + 
		\mathbf{d}^{\top}\mathbf{y}\Big\}\\
		\text{s.t. } 
		& \mathbf{H} \mathbf{x} + 
		\mathbf{G} \mathbf{y} \leq \mathbf{h} \label{cons: coupling} \\
		& \mathbf{L} \mathbf{x} + 
		\mathbf{F} \mathbf{y} 
		\leq
		\mathbf{f} \label{cons: primal feasibility} \\
		& \boldsymbol{\lambda} \geq \mathbf{0} \label{cons: dual feasibility 1}\\
		& \mathbf{g} + \mathbf{F}^\top \boldsymbol{\lambda} = \mathbf{0} \label{cons: dual feasibility 2}\\
		& (\mathbf{f} - \mathbf{L} \mathbf{x} - 
		\mathbf{F} \mathbf{y})^\top \boldsymbol{\lambda} = 0. \label{cons: cs}
	\end{align}
\end{subequations}
Here, primal feasibility of the follower is guaranteed by constraint (\ref{cons: primal feasibility}), constraints (\ref{cons: dual feasibility 1})–(\ref{cons: dual feasibility 2}) ensure dual feasibility and stationarity, respectively, whereas (\ref{cons: cs}) enforces complementary slackness.

To linearize (\ref{cons: cs}), one may introduce new binary variables $\mathbf{u} \in \{0, 1\}^{m_f}$ and primal slack variables $\mathbf{s} := \mathbf{f} - \mathbf{L} \mathbf{x} - 
\mathbf{F} \mathbf{y}$ such~that
 \begin{align} \nonumber
 	& \boldsymbol{\lambda} \leq M^{d} \mathbf{u} \; \mbox{ and } \; 
	\mathbf{s} \leq M^{p} (\mathbf{1} - \mathbf{u}).
\end{align}
Here, $M^d \in \mathbb{R}_{>0}$ and $M^p \in \mathbb{R}_{>0}$ are big-$M$ constants representing global upper bounds, respectively, for the dual variables $\boldsymbol{\lambda}$ and the primal slacks $\mathbf{s}$; see, e.g., \cite{Audet1997, Kleinert2020}. Consequently, [\textbf{BLP}] can be solved by leveraging a mixed-integer linear programming (MILP) problem of the form:
\begin{subequations} \label{MILP} \begin{align}
		 \min_{\mathbf{x}, \mathbf{y}, \boldsymbol{\lambda}, \mathbf{s}, \mathbf{u}} \; &  	\Big\{ \mathbf{a}^{\top}\mathbf{x} + 
		 \mathbf{d}^{\top}\mathbf{y} \Big\} \\ 
		 \text{s.t. } & \text{(\ref{cons: coupling})--(\ref{cons: dual feasibility 2})}, \\ 
		 & \mathbf{s} = \mathbf{f} - \mathbf{L} \mathbf{x} - \mathbf{F} \mathbf{y} \\ & \boldsymbol{\lambda} \leq M^{d} \mathbf{u} \label{cons: big M 0 dual}\\
		 & \mathbf{s} \leq M^{p} (\mathbf{1} - \mathbf{u}) \label{cons: big M 0 primal} \\
		 & \mathbf{u} \in \{0, 1\}^{m_f}.
	 \end{align} \end{subequations}

From a practical perspective, the choice of the big-$M$ parameters, $M^p$ and $M^d$, affects both the correctness of the MILP  formulation  \color{black} (\ref{MILP}) and its computational performance. Intuitively, choosing correct big-$M$\hspace{-0.5pt}s in (\ref{MILP}) is difficult 
since  at least one of the primal and dual feasible regions defined by constraints (\ref{cons: primal feasibility}) and (\ref{cons: dual feasibility 1})--(\ref{cons: dual feasibility 2}), respectively, \color{black} is unbounded; see, e.g.,~\cite{Clark1961}. As a consequence, practitioners often select sufficiently large big-$M$ parameters without explicitly verifying their correctness, solve (\ref{MILP}) using off-the-shelf MILP solvers, and then treat the obtained solution as \textit{bilevel optimal}. 

\looseness-1 Indeed, in terms of computational complexity, Buchheim~\cite{Buchheim2023} shows that 
big-$M$\hspace{-0.5pt}s of polynomial encoding length in~(\ref{cons: big M 0 dual})-(\ref{cons: big M 0 primal})  that preserve at least one bilevel-optimal solution \color{black} can be computed in polynomial~time. However, while using these big-$M$ values in (\ref{MILP})  preserves the optimal objective value of [\textbf{BLP}], \color{black} they are impractically large and may~compromise computational effectiveness. Moreover, Kleinert et al.~\cite{Kleinert2020} demonstrate that verifying whether  a given dual big-$M$ parameter $M^d$ \color{black} preserves either feasible or optimal solutions of the follower’s  dual \color{black} problem 
is computationally difficult. In~particular,
checking that no feasible vertex of the follower's dual polyhedron is cut off is $\mathrm{coNP}$-complete, whereas  verifying that, for every feasible leader's decision, at least one optimal dual solution of the follower is preserved \color{black} is as hard as solving the original bilevel~problem.

It is worth mentioning that Kleinert et al.~\cite{Kleinert2020} analyze  the properties of $M^d$ \color{black} \textit{a priori}, i.e., before solving the MILP  formulation \color{black} (\ref{MILP}). At the same time, one might hope that, after solving~(\ref{MILP}), it is possible to verify efficiently whether the obtained solution is bilevel optimal or whether the previously chosen big-$M$ constants  $M^p$ and $M^d$ are \textit{bilevel-correct}, i.e., the MILP formulation~(\ref{MILP}) preserves all bilevel-optimal solutions of~[\textbf{BLP}]. \color{black}
That is, we consider the following \textit{a posteriori}~questions:
\begin{itemize}
	\item[\textbf{Q1}.] Given potentially incorrect $M^p$ and $M^d$, is it possible to verify in polynomial time that an optimal solution of~(\ref{MILP}) is bilevel optimal, i.e., solves the original bilevel problem [\textbf{BLP}]?
	\item[\textbf{Q2}.] Does the verification  of bilevel correctness of the pair $(M^p, M^d)$ \color{black} become efficient when performed a posteriori, i.e., when an optimal solution to~(\ref{MILP}) is available?
\end{itemize}

From a computational complexity perspective, \textbf{Q1} and \textbf{Q2} are fundamentally different questions. Global big-$M$ correctness in \textbf{Q2} concerns properties of the big-$M$ parameters, whereas bilevel optimality in \textbf{Q1} concerns individual solutions. In particular,  assuming that [\textbf{BLP}] admits an optimal solution, bilevel correctness of the pair ($M^p$, $M^d$) \color{black} implies bilevel optimality  of any optimal solution of~(\ref{MILP}), \color{black} but not vice versa. Moreover, a positive answer to either \textbf{Q1} or \textbf{Q2} could help tune the parameters $M^p$ and $M^d$ in practice. For example, one could iteratively solve (\ref{MILP}) and verify bilevel optimality or big-$M$ correctness, following a trial-and-error procedure similar to that discussed~in~\cite{Pineda2018,Pineda2019}.


In addition, we note that the MILP formulation~(\ref{MILP}) is always a restriction of the corresponding KKT formulation~(\ref{MPEC}), i.e., the optimal objective value of (\ref{MILP}) yields an upper bound on that of~[\textbf{BLP}]. Consequently,~\textbf{Q1} also asks whether the chosen big-$M$ values preserve the optimal value of~[\textbf{BLP}], or equivalently, at least one bilevel-optimal solution. Indeed, the given optimal solution of~(\ref{MILP}) is bilevel optimal if and only if the optimal objective function values of~(\ref{MILP}) and~[\textbf{BLP}] coincide. 
\color{black}

It turns out, however, that the problems in \textbf{Q1} and \textbf{Q2} are both computationally difficult. Specifically, we establish the following negative results:
\begin{itemize}
	\item To address \textbf{Q1}, we prove that, even when $M^d = M^p$, the problem of verifying that an optimal solution of~(\ref{MILP}) is bilevel optimal is~$\mathrm{coNP}$-complete  (Theorem \ref{theorem 1}). \color{black} 
	\item 
	This result extends to strong duality-based reformulations of mixed-integer BLPs  (Theorem \ref{theorem 2}) \color{black} and to the class of min-max problems  (Theorems \ref{theorem 3} and \ref{theorem 4}). \color{black}
	\item As a consequence, we further show that,  in most of the considered problem settings,  it is $\mathrm{coNP}$-hard to verify whether a given optimal solution of (\ref{MILP}) is an approximate solution of the original bilevel problem (Corollaries \ref{corollary 1}--\ref{corollary 3}). \color{black}
	\item To address \textbf{Q2}, we demonstrate that the hardness results of~\cite{Kleinert2020} persist in the a posteriori setting, where an optimal solution of the MILP  formulation \color{black} (\ref{MILP}) is part of the input  (Theorems~\ref{theorem 5} and~\ref{theorem 6}).
	\item As a consequence, we show that verifying bilevel correctness of the pair $(M^p, M^d)$ is $\mathrm{coNP}$-hard (Corollary \ref{corollary 4}). 
\end{itemize}
 Notably, the reductions in Theorems~\ref{theorem 3} and~\ref{theorem 6} do not establish strong $\mathrm{coNP}$-hardness. \color{black}

We therefore conclude that access to an optimal solution of a big-$M$  MILP formulation \color{black} does not necessarily alleviate the difficulty  of verifying whether the chosen big-$M$ values preserve the optimal value or all bilevel-optimal solutions of the original bilevel problem. In particular, even after solving the MILP formulation to optimality, there is no computationally tractable worst-case procedure for certifying that the chosen big-$M$ values are adequate.
\color{black}

The remainder of the paper is organized as follows.  Sections~\ref{sec: KKT}--\ref{sec: min-max} address the research question~\textbf{Q1}. \color{black} In particular, Sections~\ref{sec: KKT} and~\ref{sec: strong duality}  analyze \color{black} KKT-based and strong-duality-based reformulations of [\textbf{BLP}], respectively, while Section~\ref{sec: min-max} extends these results to min-max problems. Section~\ref{sec: global}  addresses the research question~\textbf{Q2}, namely verification of  bilevel \color{black} correctness \color{black} for the MILP formulation~(\ref{MILP}).  Finally, Section~\ref{sec: conclusion} concludes the paper. \color{black}


\textbf{Notation.}  We use $\mathbb{R}_+$, $\mathbb{Q}_+$, and $\mathbb{Z}_+$ to denote the sets of nonnegative real, rational, and integer numbers, respectively. Similarly, $\mathbb{R}_{>0}$, $\mathbb{Q}_{>0}$, and $\mathbb{Z}_{>0}$ are the corresponding sets of positive numbers. \color{black}
Vectors and matrices are denoted by boldface letters, with~
$\mathbf{1}$ representing 
the all-ones vector of appropriate dimension. Finally, $\Vert \cdot \Vert_p$ denotes the $\ell_p$-norm with $p \geq 1$,  and $[n]$ denotes the set~$\{1,\ldots,n\}$\color{black}.
\section{KKT-based reformulation} \label{sec: KKT}
In the remainder of the paper, the leader's  projected \color{black} feasible region in [\textbf{BLP}] is defined as:
 \begin{equation} \label{eq: leader's feasible set extended}
	 H_x := \big\{\mathbf{x} \in \mathbb{R}^{n_l}: \; \exists \mathbf{y} \in \mathbb{R}^{n_f} \; \text{ such that } \; 
	 \mathbf{H} \mathbf{x} + \mathbf{G} \mathbf{y} \leq \mathbf{h} \;
	 \text{ and } \;
	 \mathbf{L} \mathbf{x} + \mathbf{F} \mathbf{y} \leq \mathbf{f} \big\}.
	 \end{equation}
 Throughout the remainder of the paper, \color{black} we make the following  standing \color{black} assumption: 
\begin{itemize}
	 	\item[\textbf{A1.}] The  projected \color{black} feasible region $H_x$ is non-empty and bounded, and the follower's problem in~(\ref{cons: follower}) admits a finite optimal solution for all $\mathbf{x} \in H_x$.
\end{itemize}

To address the research question \textbf{Q1}, we first analyze the KKT-based big-$M$  formulation \color{black} (\ref{MILP}) and introduce the following decision problem:
\begin{itemize}
	\item[$ $] [\textbf{BLP-D}]:  Given an instance of [\textbf{BLP}] with rational data, rational parameters
	$M^d,M^p\in\mathbb{Q}_{>0}$, and the pair $(\tilde{\mathbf{x}}^*,\tilde{\mathbf{y}}^*)$
	corresponding to an optimal solution of the MILP formulation~(\ref{MILP}), \color{black} decide whether $(\tilde{\mathbf{x}}^*, \tilde{\mathbf{y}}^*)$ is also optimal for [\textbf{BLP}].
\end{itemize}
 Here, [\textbf{BLP-D}] is treated as a \textit{promise problem}, where a solution of the corresponding big-$M$ MILP formulation~(\ref{MILP}) is provided as part of the input and is promised to be optimal; see, e.g.,~\cite{Goldreich2006}. \color{black}

To establish $\mathrm{coNP}$-completeness, we use a reduction from the optimality verification problem for integer linear programs (ILPs):
\begin{itemize}
	\item[$ $] [\textbf{ILP-D}]: Given a 0–1 integer linear program
	\begin{equation} \label{ILP}
		\min \Big\{\mathbf{c}^\top \mathbf{x}: \; \mathbf{A}\mathbf{x} \leq \mathbf{b}, \; \mathbf{x} \in \{0, 1\}^n \Big\},
	\end{equation}
	where $(\mathbf{A}, \mathbf{b}, \mathbf{c})$ are integral, and a feasible solution $\bar{\mathbf{x}} = \mathbf{0}$, decide whether $\bar{\mathbf{x}}$ is optimal for (\ref{ILP}).
\end{itemize}

Since the complement of [\textbf{ILP-D}] for an arbitrary feasible decision $\bar{\mathbf{x}} \in \{0, 1\}^n$ is $\mathrm{NP}$-complete~\cite{Garey1979}, the problem is $\mathrm{coNP}$-complete when $\bar{\mathbf{x}}$ is part of the input. Furthermore, this result remains valid for the special case where~$\bar{\mathbf{x}}=\mathbf{0}$. That is, for each $i \in [n]$, one may introduce new variables~$x'_i$ by setting $x'_i := x_i$ when $\bar{x}_i = 0$, and $x'_i := 1 - x_i$ otherwise. This polynomial-time bijection maps $\bar{\mathbf{x}}$ to $\mathbf{0}$ and translates~(\ref{ILP}) into an equivalent 0-1 ILP (up to an additive constant in the objective function). 

The following result establishes $\mathrm{coNP}$-completeness of [\textbf{BLP-D}].
\begin{theorem} \label{theorem 1}
The optimality verification problem \upshape [\textbf{BLP-D}] \itshape is $\mathrm{coNP}$-complete even when there is only one global big-$M$ parameter, i.e., $M^p = M^d$. 
\begin{proof}

The problem [\textbf{BLP-D}] belongs to $\mathrm{coNP}$: for a ``no'' instance, a certificate is a rational pair $(\mathbf{x},\mathbf{y})$ of polynomial encoding length such that $(\mathbf{x},\mathbf{y})$ is bilevel feasible and
\[
\mathbf{a}^{\top}\mathbf{x}+\mathbf{d}^{\top}\mathbf{y}
<
\mathbf{a}^{\top}\tilde{\mathbf{x}}^*
+\mathbf{d}^{\top}\tilde{\mathbf{y}}^* .
\]
Indeed, bilevel feasibility can be verified in polynomial time by solving the follower's linear program in~(\ref{cons: follower}) for the given $\mathbf{x}$, verifying that the supplied $\mathbf{y}$ is optimal, and ensuring that the coupling constraints~(\ref{cons: leader}) hold. The strict objective inequality can then be checked in polynomial~time.
\color{black}	

To prove $\mathrm{coNP}$-completeness, we use a polynomial-time reduction from [\textbf{ILP-D}]. First, we observe that (\ref{ILP}) can be expressed as an  instance of [\textbf{BLP}] \color{black} of the form:
\begin{subequations} \label{ILP 2}
	\begin{align}
		\min_{\mathbf{x}, \mathbf{y}^*} & \; \mathbf{c}^\top \mathbf{x} \\
		\mbox{s.t. } & \mathbf{A}\mathbf{x} \leq \mathbf{b} \label{cons: leader feasibility} \\
		& \mathbf{0} \leq \mathbf{x} \leq \mathbf{1} \\
		& \mathbf{y}^* = \mathbf{0} \label{cons: coupling binary} \\
		& \mathbf{y}^* \in \argmin_{\,\mathbf{y}}
		\Big\{ -\sum_{i = 1}^{n} \kappa\color{black} y_i \Big\} \label{obj: follower binary} \\
		& \text{s.t.} \quad \kappa\color{black}\mathbf{y} \leq \mathbf{x} \label{cons: follower binary 1}\\
		& \phantom{\text{s.t.} \quad} \mathbf{y} \leq \mathbf{1} - \mathbf{x} \label{cons: follower binary 2}, 
	\end{align}
\end{subequations}
where $\kappa > 1$ is any fixed rational constant of polynomial encoding length.

Notably, in contrast to the standard optimistic reformulation in \cite{Audet1997}, the follower's objective function in (\ref{obj: follower binary}) and the follower constraints (\ref{cons: follower binary 1}) are multiplied~by~$\kappa$. This scaling preserves equivalence between~(\ref{ILP}) and (\ref{ILP 2}). Indeed,
\color{black}
constraints~(\ref{obj: follower binary})--(\ref{cons: follower binary 2})
together with the coupling constraint~(\ref{cons: coupling binary}) imply that the follower's optimal solution satisfies
\[
y_i^* = \min\{\frac{1}{\kappa}\color{black}x_i, 1 - x_i\} = 0
\quad \forall i \in [n],
\]
and therefore $\mathbf{x} \in \{0,1\}^n$.
As we demonstrate below, the choice of $\kappa$ allows us to adjust the big-$M$ values in the resulting big-$M$ MILP formulation of (\ref{ILP 2}). Finally, the obtained instance of [\textbf{BLP}] has polynomial encoding size and, since $\bar{\mathbf{x}} = \mathbf{0}$ is feasible for (\ref{ILP}), it satisfies Assumption~\textbf{A1}.
\color{black}

Let $\boldsymbol{\lambda} \in \mathbb{R}^n_+$ and $\boldsymbol{\nu} \in \mathbb{R}^n_+$ be dual variables corresponding to the follower constraints (\ref{cons: follower binary 1})--(\ref{cons: follower binary 2}), respectively. Applying the KKT conditions to the follower's problem (\ref{obj: follower binary})--(\ref{cons: follower binary 2}) yields a version of~(\ref{MILP}).  Specifically, using the coupling constraint~(\ref{cons: coupling binary}), we obtain the following big-$M$ formulation of~(\ref{ILP 2}): \color{black}
\begin{subequations} \label{ILP 3}
	\begin{align}
		\min_{\mathbf{x}, \boldsymbol{\lambda}, \boldsymbol{\nu}, \mathbf{u}, \mathbf{v}} & \; \mathbf{c}^\top \mathbf{x} \\
		\mbox{s.t. } & \mathbf{A}\mathbf{x} \leq \mathbf{b} \\
		& \begin{rcases} 
			\kappa\lambda_i + \nu_i = \kappa\color{black} \\ 
			0 \leq \lambda_i \leq M^d \, u_i \\
			0 \leq x_i \leq M^p (1 - u_i) \\
			0 \leq \nu_i \leq M^d \, v_i \\
			0 \leq 1 - x_i \leq M^p (1 - v_i) \quad
		\end{rcases} \; \forall i \in [n] \\
		& \mathbf{u} \in \{0, 1\}^n, \; \mathbf{v} \in \{0, 1\}^n.
	\end{align}
\end{subequations}
 In particular, $\mathbf{y}^*=\mathbf{0}$ implies that the primal slacks corresponding to follower constraints~(\ref{cons: follower binary 1})--(\ref{cons: follower binary 2}) are $\mathbf{x}$ and $\mathbf{1}-\mathbf{x}$, respectively, allowing $\mathbf{y}^*$ to be eliminated from the formulation.

We then note that feasibility of (\ref{ILP 3}) requires that $\mathbf{u} + \mathbf{v} = \mathbf{1}$.
Indeed, $u_i=v_i=0$, $i \in [n]$, contradicts $\kappa\lambda_i+\nu_i=\kappa$, whereas $u_i=v_i=1$ contradicts the two primal-slack constraints.
\color{black}
Eliminating $\boldsymbol{\nu}$ and $\mathbf{v}$ therefore yields the~following formulation:\color{black}
\begin{subequations} \label{ILP 4}
	\begin{align}
		\min_{\mathbf{x}, \boldsymbol{\lambda}, \boldsymbol{u}} & \; \mathbf{c}^\top \mathbf{x} \\
		\mbox{s.t. } & \mathbf{A}\mathbf{x} \leq \mathbf{b} \\
		& \begin{rcases} 
			0 \leq \lambda_i \leq M^d \, u_i \\
			0 \leq x_i \leq M^p (1 - u_i) \\
			0 \leq 1 - \lambda_i \leq \frac{1}{\kappa}\color{black} M^d (1 - u_i) \quad \\
			0 \leq 1 - x_i \leq M^p \, u_i
		\end{rcases} \; \forall i \in [n] \label{cons: big M}\\
		& \mathbf{u} \in \{0, 1\}^n. \label{cons: binary big M}
	\end{align}
\end{subequations}

 Finally, \color{black} we set $M^p = M^d = 1$. In this case, $u_i = 0$ implies that $\lambda_i = 0$ and
$1 - \lambda_i \leq \tfrac{1}{\kappa}$, which is infeasible since $\kappa>1$.
\color{black}
Hence, constraints (\ref{cons: big M})
imply a unique feasible (and optimal) solution of~(\ref{ILP 4}) with $u_i = 1$, $x_i = 0$ and $\lambda_i = 1$ for each $i \in [n]$.
We conclude that any optimal solution of the big-$M$ formulation of (\ref{ILP 2}) satisfies $\tilde{\mathbf{x}}^*=\mathbf{0}$ and~$\tilde{\mathbf{y}}^*=\mathbf{0}$.
Furthermore, $(\tilde{\mathbf{x}}^*,\tilde{\mathbf{y}}^*)$ is bilevel optimal for~(\ref{ILP 2}) if and only if $\tilde{\mathbf{x}}^*$ is optimal for~(\ref{ILP}).  Hence, [\textbf{ILP-D}] is a ``yes''-instance if and only if [\textbf{BLP-D}] is a ``yes''-instance, and the result follows. \color{black} 
\end{proof}	
\end{theorem}

Importantly, Theorem~\ref{theorem 1} has a practical implication for iterative procedures that repeatedly increase the big-$M$ parameters and resolve the resulting big-$M$ MILP formulation (\ref{MILP}). To this end, we make the following remark: 
\begin{remark} \upshape
Let $K \in \mathbb{Z}_{>0}$ be polynomially bounded in the input size. By setting $\kappa = 2^{K+1}$ in the construction of Theorem~\ref{theorem 1}, the same optimal solution of (\ref{ILP 3}) is preserved for all successively doubled common big-$M$ values
$M^p = M^d \in \{2, 4, \ldots, 2^K\}$.
Thus, repeatedly increasing a common big-$M$ parameter and resolving the formulation (\ref{MILP}) does not, in general, make a posteriori bilevel-optimality verification tractable. \hfill$\square$
\end{remark}
\color{black}

In the following, we establish a stronger hardness result in the sense of approximation. Let $z^*$ and~$z^{\text{max}}$ denote, respectively, the optimal objective function value and  the maximum objective function value over all feasible solutions \color{black} of~(\ref{ILP}).  We~then define an $\varepsilon$-approximate optimality verification problem for (\ref{ILP}) in the sense of Vavasis~\cite{Vavasis1993}:
\begin{itemize}
	\item[$ $] [\textbf{ILP-D}$_\varepsilon$]: Given an instance of~(\ref{ILP})  with integral data, \color{black} a feasible solution $\bar{\mathbf{x}} = \mathbf{0}$, and a  rational \color{black} parameter $\varepsilon \in (0, 1)$, decide whether
	\begin{equation} \label{eq: approximation}
	\mathbf{c}^\top \bar{\mathbf{x}} - z^* = -z^* \leq \varepsilon (z^{\text{max}} - z^*).	
  \end{equation}	
\end{itemize}
Since $z^{\text{max}}$ is defined as the maximum feasible  objective function \color{black} value in (\ref{ILP}), membership of [\textbf{ILP-D}$_\varepsilon$] in $\mathrm{coNP}$ is not immediate. We therefore restrict our attention to
proving $\mathrm{coNP}$-hardness. 
\begin{lemma} \label{lemma 1} 
\upshape [\textbf{ILP-D}$_\varepsilon$] \itshape is $\mathrm{coNP}$-hard for any fixed $\varepsilon \in (0, 1)$.
\begin{proof}
Since all data in (\ref{ILP}) are integral, $\bar{\mathbf{x}} = \mathbf{0}$ is not optimal for (\ref{ILP}) if and only if there exists feasible~$\mathbf{x}$ such that $\mathbf{c}^\top \mathbf{x} \leq -1$. Let us consider the following 0-1 ILP associated with~(\ref{ILP}):
\begin{subequations} \label{ILP epigraph}
	\begin{align}
		\min_{\mathbf{x}, z} & \, -z \\
		\text{s.t. } &
		\mathbf{A}\mathbf{x} \leq \mathbf{b} \\
		& \mathbf{c}^\top \mathbf{x} \leq -1 + M(1 - z) \\
		& \mathbf{\mathbf{x}} \in \{0,1\}^n, \; z \in \{0, 1\}, 
	\end{align}
\end{subequations}
where $M = n \Vert \mathbf{c} \Vert_{\infty} + 1$. Notably, if $\bar{\mathbf{x}} = \mathbf{0}$ is optimal for (\ref{ILP}), then $z = 1$ is infeasible in (\ref{ILP epigraph}), and hence the optimal objective function value of (\ref{ILP epigraph}) is $0$. Otherwise, if $\bar{\mathbf{x}} = \mathbf{0}$ is not optimal for (\ref{ILP}), then $z = 1$ is optimal for (\ref{ILP epigraph}), yielding the objective value $-1$.  Thus, the constructed instance in (\ref{ILP epigraph}) has optimal objective function value either $0$ or $-1$, while $z^{\mathrm{max}}=0$, and inequality~(\ref{eq: approximation}) holds in the former case and fails in the latter for any fixed $\varepsilon \in (0,1)$. \color{black} This observation concludes the~proof. 
\end{proof}
\end{lemma}

Let $z_b^*$ and $z^{\text{max}}_b$ denote the optimal and the maximum feasible objective function values of [\textbf{BLP}],  respectively, whenever both are finite. \color{black} Then, an $\varepsilon$-approximate optimality verification problem for [\textbf{BLP}] is defined as:
\begin{itemize}
	\item[$ $] [\textbf{BLP-D}$_\varepsilon$]:  Given the input of [\textbf{BLP-D}] and a rational parameter $\varepsilon \in (0, 1)$, \color{black} decide whether $\mathbf{a}^{\top} \tilde{\mathbf{x}}^* + 
	\mathbf{d}^{\top} \tilde{\mathbf{y}}^* - z_b^* \leq \varepsilon (z^{\text{max}}_b - z^*_b)$.
\end{itemize}
The following result demonstrates that even deciding whether the MILP  formulation \color{black} (\ref{MILP}) yields an $\varepsilon$-approximation of [\textbf{BLP}] is $\mathrm{coNP}$-hard.
\begin{corollary} \label{corollary 1}
Under the conditions of Theorem \ref{theorem 1}, \upshape [\textbf{BLP-D}$_\varepsilon$] \itshape is $\mathrm{coNP}$-hard for every fixed  rational\color{black}~$\varepsilon\in\nobreak(0, 1)$.
\begin{proof} The result follows from Lemma \ref{lemma 1} and the proof of Theorem \ref{theorem 1}.  In particular, the construction in Theorem~\ref{theorem 1} preserves all feasible objective function values, and hence $z_b^*$ and $z_b^{\max}$ coincide with the corresponding values of~(\ref{ILP epigraph}). \color{black} 
\end{proof}
\end{corollary}
\section{Strong duality-based reformulation} \label{sec: strong duality}
In this section, we consider a class of mixed-integer optimistic BLPs, where the leader variables are assumed to be \textit{binary} and the follower variables are continuous; see, e.g., \cite{Zare2019}. In other words, we consider the following optimistic bilevel problem:
\begin{subequations}
	\label{MIBLP}
	\begin{align}
		[\textbf{BLP}']: \quad
		\min_{\mathbf{x}, \mathbf{y}^*} \; &  
		\Big\{ \mathbf{a}^{\top}\mathbf{x} + 
		\mathbf{d}^{\top}\mathbf{y}^* \Big\}
		\label{obj: leader binary} \\
		\text{s.t. } 
		& \mathbf{x} \in \{0, 1\}^{n_l} \label{cons: leader binary 1} \\
		& \mathbf{H} \mathbf{x} + 
		\mathbf{G} \mathbf{y}^* \leq \mathbf{h} \label{cons: leader binary 2} \\
		& \mathbf{y}^* \in 
		\argmin_{\,\mathbf{y}} \,
		\Big\{\mathbf{g}^{\top}\mathbf{y}: \; \mathbf{L} \mathbf{x} + 
		\mathbf{F} \mathbf{y} 
		\leq
		\mathbf{f} \Big\}, \label{cons: follower binary}
	\end{align}
\end{subequations}
where, in contrast to [\textbf{BLP}], the additional binary constraint (\ref{cons: leader binary 1}) is imposed.

As argued in \cite{Zare2019}, when the leader variables $\mathbf{x}$ are integer, it is often computationally advantageous to replace the complementary slackness conditions (\ref{cons: cs}) with a single strong duality constraint. We~therefore focus on the following strong duality-based reformulation of [\textbf{BLP}$'$]:
\begin{subequations}
	\label{MPEC binary}
	\begin{align}
		\min_{\mathbf{x}, \mathbf{y}, \boldsymbol{\lambda}} \; &
			\Big\{ \mathbf{a}^{\top}\mathbf{x} + 
		\mathbf{d}^{\top}\mathbf{y} \Big\} \\
		\text{s.t. } 
		& \text{(\ref{cons: coupling})--(\ref{cons: dual feasibility 2})}, \label{cons: dual binary} \\
		& \mathbf{x} \in \{0, 1\}^{n_l} \label{cons: leader binary} \\
		& \mathbf{g}^\top \mathbf{y} = (\mathbf{L} \mathbf{x} - \mathbf{f})^\top \boldsymbol{\lambda}, \label{cons: strong duality}
	\end{align}
\end{subequations}
where  the equality \color{black} constraint (\ref{cons: strong duality}) ensures that the optimal primal and dual objective function values in (\ref{cons: follower binary}) coincide. Next, for each $i \in [n_l]$ and $j \in [m_f]$, we introduce a new variable $u_{ij} = x_i \lambda_j$.  Assuming that \color{black} $\mathbf{x} \in \{0, 1\}^{n_l}$, the equality constraint $u_{ij} = x_i \lambda_j$ can be replaced by the following linear inequalities \cite{Zare2019}:
\begin{subequations} 
	\begin{align}
		& 0 \leq u_{ij} \leq \lambda_j \label{cons: stong duality linear 1} \\
		& u_{ij} \leq M_j x_{i} \label{cons: stong duality linear 2}\\
		& u_{ij} \geq \lambda_j + M_j (x_{i} - 1), \label{cons: stong duality linear 3}
	\end{align}
\end{subequations}
where $M_j  \in \mathbb{R}_+\color{black}$ is an upper bound for $\lambda_j$. As a result, [\textbf{BLP}$'$] admits the following  big-$M$ MILP formulation: \color{black}
\begin{subequations}
	\label{MILP 2}
	\begin{align}
		\min_{\mathbf{x}, \mathbf{y}, \boldsymbol{\lambda}, \mathbf{u}} \; & 
			\Big\{ \mathbf{a}^{\top}\mathbf{x} + 
		\mathbf{d}^{\top}\mathbf{y} \Big\} \\
		\text{s.t. } 
		& \text{(\ref{cons: dual binary})--(\ref{cons: leader binary})},\\
		& \text{(\ref{cons: stong duality linear 1})--(\ref{cons: stong duality linear 3})} \quad \forall i \in [n_l], \quad \forall j \in [m_f], \\
		& \mathbf{g}^\top \mathbf{y} = \sum_{i = 1}^{n_l} \sum_{j = 1}^{m_f}  L_{ji} \color{black} u_{ij} - \mathbf{f}^\top \boldsymbol{\lambda}. \label{cons: strong duality 2}
	\end{align}
\end{subequations}

Similar to [\textbf{BLP-D}], we formulate an optimality verification problem for [\textbf{BLP}$'$]:
\begin{itemize}
		\item[$ $] [\textbf{BLP-D$'$}]:  Given an instance of [\textbf{BLP}$'$] with rational data, rational parameters $\mathbf{M} \in \mathbb{Q}^{m_f}_{>0}$, and the pair $(\tilde{\mathbf{x}}^*, \tilde{\mathbf{y}}^*)$ corresponding to an optimal solution of the MILP formulation~(\ref{MILP 2}), \color{black} decide whether $(\tilde{\mathbf{x}}^*, \tilde{\mathbf{y}}^*)$ is also optimal for [\textbf{BLP}$'$].
	\end{itemize}
The following result holds.

\begin{theorem} \label{theorem 2}
	The optimality verification problem \upshape [\textbf{BLP-D}$'$] \itshape is $\mathrm{coNP}$-complete even when there is only one global big-$M$ parameter, i.e., $M_j = M$ for all $j \in [m_f]$. 
	\begin{proof}
		 Membership in $\mathrm{coNP}$ follows analogously to Theorem \ref{theorem 1}. To prove $\mathrm{coNP}$-completeness, \color{black} we use a polynomial-time reduction from [\textbf{ILP-D}] and consider the optimistic bilevel problem (\ref{ILP 2}) with an additional binary constraint $\mathbf{x} \in \{0, 1\}^n$. In~particular, this constraint is redundant in (\ref{ILP 2}), since it is readily enforced by constraints (\ref{cons: coupling binary})--(\ref{cons: follower binary 2}). By applying the strong duality-based reformulation to the follower's problem (\ref{obj: follower binary})--(\ref{cons: follower binary 2}) with $\mathbf{y}^* = \mathbf{0}$, we obtain the following  big-$M$ MILP formulation \color{black} of (\ref{ILP 2}):
		\begin{subequations} \label{ILP 3 2}
			\begin{align}
				\min_{\mathbf{x}, \boldsymbol{\lambda}, \boldsymbol{\nu}, \mathbf{u}, \mathbf{v}} & \; \mathbf{c}^\top \mathbf{x} \\
				\mbox{s.t. } & \mathbf{x} \in \{0, 1\}^n \label{cons: binary} \\ & \mathbf{A}\mathbf{x} \leq \mathbf{b} \\
				& \begin{rcases}  \kappa \lambda_i + \nu_i = \kappa \\ 
				0 \leq u_i \leq \lambda_i \\
			  0 \leq v_i \leq \nu_i \\
			  u_i \leq M_i x_i \\
			  v_i \leq \tilde{M}_i x_i \\
			  u_i \geq \lambda_i + M_i(x_i - 1) \\
			  v_i \geq \nu_i + \tilde{M}_i(x_i - 1) \quad
				\end{rcases} \; \forall i \in [n] \label{cons: strong duality linear} \\
				& -\sum_{i = 1}^n \big(u_i + \nu_i - v_i \big) = 0, \label{cons: strong duality 3} 
			\end{align}
		\end{subequations}
		 where constraints $u_i = \lambda_i x_i$ and $v_i = \nu_i x_i$ are linearized for each $i \in [n]$.  Furthermore, $M_i$ and $\tilde{M}_i$, $i \in [n]$, denote the components of the big-$M$ vector $\mathbf{M}$ associated with $\lambda_i$ and $\nu_i$, respectively. \color{black}
		 		
		Then, $u_i \geq 0$, $\nu_i - v_i \geq 0$ and (\ref{cons: strong duality 3}) imply that $u_i = 0$ and $v_i = \nu_i$ for each~$i \in [n]$. By~eliminating $u_i$, $v_i$ and $\nu_i$, (\ref{ILP 3 2}) can be expressed as: 
		\begin{subequations} \label{ILP 4 2}
			\begin{align}
				\min_{\mathbf{x}, \boldsymbol{\lambda}} & \; \mathbf{c}^\top \mathbf{x} \\
				\mbox{s.t. } & \mathbf{x} \in \{0, 1\}^n \\ & \mathbf{A}\mathbf{x} \leq \mathbf{b} \\
				& \begin{rcases} 
					0 \leq \lambda_i \leq 1 \\
					 1 - \lambda_i \leq \frac{1}{\kappa} \tilde{M}_i x_i \\
					0 \geq \lambda_i + M_i(x_i - 1) \quad
				\end{rcases} \; \forall i \in [n]. \label{cons: strong duality linear 2} 
			\end{align}
		\end{subequations}
		
		Finally, we set $M_i = \tilde{M}_i = 1$ for all $i \in [n]$. Analogously to the proof of Theorem \ref{theorem 1}, we conclude that (\ref{ILP 4 2}) has a unique feasible solution such that $\mathbf{x} = \mathbf{0}$ and~$\boldsymbol{\lambda} = \mathbf{1}$, and the result follows. 
	\end{proof}	
\end{theorem}

Let $z_b^{\prime*}$ and $z^{\prime \, \text{max}}_b$ be the optimal and the maximum feasible objective function values of [\textbf{BLP}$'$],  respectively, whenever both are finite. \color{black} 
Then, an $\varepsilon$-approximate optimality verification problem for [\textbf{BLP-D}$'$] reads as:

\begin{itemize}
	\item[$ $] [\textbf{BLP-D}$'_\varepsilon$]:  Given the input of [\textbf{BLP-D}$'$] and a rational parameter $\varepsilon \in (0, 1)$, \color{black} decide whether $\mathbf{a}^{\top} \tilde{\mathbf{x}}^* + 
	\mathbf{d}^{\top} \tilde{\mathbf{y}}^* - z_b^{\prime*} \leq \varepsilon(z^{\prime\, \text{max}}_b - z_b^{\prime*})$.
\end{itemize}
Similar to Corollary \ref{corollary 1}, the following result holds. 
\begin{corollary} \label{corollary 2}
	Under the conditions of Theorem \ref{theorem 2}, \upshape [\textbf{BLP-D}$'_\varepsilon$] \itshape is $\mathrm{coNP}$-hard for every fixed  rational\color{black}~$\varepsilon\in\nobreak(0, 1)$.
\end{corollary}

\section{Analysis of min-max case} \label{sec: min-max}
By construction, both Theorems \ref{theorem 1} and \ref{theorem 2} are obtained for the optimistic bilevel problem [\textbf{BLP}] with coupling constraints, i.e., $\mathbf{G} \neq \mathbf{0}$. Meanwhile, most practical BLPs do not involve coupling constraints; see, e.g., the surveys in \cite{Dempe2002, Kleinert2021} and the references therein. Motivated by this observation,  we demonstrate that our results remain valid for the restricted class of min-max BLPs, where $\mathbf{G} = \mathbf{0}$ and~$\mathbf{d} = -\mathbf{g}$. \color{black} Notably, min-max problems constitute the simplest class of BLPs, where both optimistic and pessimistic formulations coincide; see, e.g., \cite{Marcotte2005, Wiesemann2013}.

\begin{theorem} \label{theorem 3}
	The optimality verification problem \upshape [\textbf{BLP-D}] \itshape is $\mathrm{coNP}$-complete even when $M^p = M^d =~1$, $\mathbf{G} = \mathbf{0}$ and $\mathbf{d} = -\mathbf{g}$.
	\begin{proof}
		Membership in $\mathrm{coNP}$ follows from Theorem~\ref{theorem 1}. To prove $\mathrm{coNP}$-hardness, similar to the proof of Theorem \ref{theorem 1}, we use a polynomial-time reduction from [\textbf{ILP-D}]. In particular, we set $\kappa = 2$ and introduce the following associated instance of [\textbf{BLP}]: 
		\begin{subequations} \label{ILP 2 3}
			\begin{align}
				\min_{\mathbf{x}, \mathbf{y}^*} \;
				& \Big\{\eta  \mathbf{c}^\top \mathbf{x}
				+ 2\sum_{i = 1}^{n} y_i^*\Big\} \label{obj: minmax leader}\\
				\mbox{s.t. } 
				& \mathbf{A}\mathbf{x} \leq \mathbf{b} \label{cons: minmax leader feasibility}\\
				& \mathbf{0} \leq \mathbf{x} \leq \mathbf{1} \label{cons: minmax leader bounds}\\
				& \mathbf{y}^* \in \argmin_{\,\mathbf{y}}
				\Big\{ -2\sum_{i = 1}^{n} y_i \Big\} \label{obj: minmax follower}\\
				& \text{s.t.} \quad 2\mathbf{y} \leq \mathbf{x} \label{cons: minmax follower 1}\\
				& \phantom{\text{s.t.} \quad} \mathbf{y} \leq \mathbf{1} - \mathbf{x}, \label{cons: minmax follower 2}
			\end{align}
		\end{subequations}
		where $\eta \in \mathbb{Q}_{>0}$ is a scaling parameter specified later.  By construction, (\ref{ILP 2 3}) constitutes a min-max instance of [\textbf{BLP}] satisfying Assumption~\textbf{A1}.
		
		Let
		\begin{equation} \label{eq: leader's feasible set}
			X := \Big\{ \mathbf{x} \in [0,1]^n : \mathbf{A}\mathbf{x} \leq \mathbf{b} \Big\}
		\end{equation}
		denote the leader's feasible set in (\ref{ILP 2 3}) and note that the follower's problem in (\ref{obj: minmax follower})--(\ref{cons: minmax follower 2}) admits a unique optimal solution  $y_i^* = \min\{\tfrac{x_i}{2}, 1 - x_i\}$ for each $i \in [n]$. Therefore, (\ref{ILP 2 3}) can be reformulated as
		\begin{equation} \label{eq: ILP concave}
				\min_{\mathbf{x} \in X} \, \Big\{ \eta  \mathbf{c}^\top \mathbf{x} + 2 \sum_{i = 1}^n \min\{\frac{x_i}{2}, 1 - x_i\}\Big\}.
		\end{equation}
	  Furthermore, since the objective function in (\ref{eq: ILP concave}) is concave in $\mathbf{x}$, its minimum over the polytope $X$ is attained at a vertex of $X$. 
	  
	 By Cramer's rule and Hadamard's inequality, there exists an integer
	 $D \in \mathbb{Z}_{>0}$, computable in polynomial time and of polynomial encoding
	 length, such that the denominator of every coordinate of every vertex
	 $\mathbf{v} \in \text{vert}(X)$ is at most $D$; see, e.g., \cite{Schrijver1986}. Hence, for every fractional vertex $\mathbf{v} \in \text{vert}(X)$, there exists $j \in [n]$ such that
	 $v_j \in [\tfrac{1}{D},1-\tfrac{1}{D}]$, and therefore
	 \[
	 2\sum_{i=1}^n \min\left\{\frac{v_i}{2},1-v_i\right\} \geq 2\min\left\{\frac{v_j}{2},1-v_j\right\}
	 \geq \frac{1}{D},
	 \]
	 where $D \geq 2$ without loss of generality. 
	 
	 We now set $\eta := \tfrac{1}{2D(\Vert \mathbf{c} \Vert_1 + 1)}$. Then, for any fractional vertex $\mathbf{v} \in \text{vert}(X)$, we have
	  \[
	 \eta \mathbf{c}^\top \mathbf{v}
	 + 2\sum_{i=1}^n \min\left\{\frac{v_i}{2},1-v_i\right\}
	 \geq
	 -\eta\Vert\mathbf{c}\Vert_1+\frac{1}{D} \geq \frac{1}{2D}
	 >0.
	 \]
	 On the other hand, $\bar{\mathbf{x}} = \mathbf{0} \in X$ and its objective value in
	 (\ref{eq: ILP concave}) is zero. Hence, no fractional vertex of~$X$ is optimal for
	 (\ref{eq: ILP concave}), and the min-max problem (\ref{ILP 2 3}) reduces to
	 \[
	 \min_{\mathbf{x} \in X \cap \{0,1\}^n}
	 \eta \mathbf{c}^\top \mathbf{x},
	 \]
	 Consequently,
	 $\mathbf{x}^* = \mathbf{0}$ and $\mathbf{y}^* = \mathbf{0}$ constitute an optimal solution for (\ref{ILP 2 3})  if and only if $\mathbf{x}^* = \mathbf{0}$ is optimal for~(\ref{ILP}).
	  
	 Next, let $\boldsymbol{\lambda} \in \mathbb{R}^n_+$ and $\boldsymbol{\nu} \in \mathbb{R}^n_+$ be dual variables corresponding to the follower constraints (\ref{cons: minmax follower 1}) and~(\ref{cons: minmax follower 2}), respectively. Applying the KKT conditions to the follower's problem in (\ref{ILP 2 3}) yields the following big-$M$ MILP formulation:
		\begin{subequations} \label{ILP 3 3}
			\begin{align}
				\min_{\mathbf{x}, \mathbf{y}, \boldsymbol{\lambda}, \boldsymbol{\nu},
					\mathbf{u}, \mathbf{v}} \;
				& \Big\{\eta \mathbf{c}^\top \mathbf{x}
				+ 2\sum_{i = 1}^{n} y_i\Big\} \label{obj: minmax MILP}\\
				\mbox{s.t. } 
				& \mathbf{x} \in X \\
				& \begin{rcases}
					2\lambda_i + \nu_i = 2\\
					0 \leq \lambda_i \leq M^d u_i\\
					0 \leq x_i - 2y_i \leq M^p(1-u_i)\\
					0 \leq \nu_i \leq M^d v_i\\
					0 \leq 1-x_i-y_i \leq M^p(1-v_i) \quad
				\end{rcases}
				\quad \forall i \in [n] \label{cons: minmax big M}\\
				& \mathbf{u} \in \{0,1\}^n,\quad
				\mathbf{v} \in \{0,1\}^n. \label{cons: minmax binaries}
			\end{align}
		\end{subequations}
		
		With $M^p = M^d = 1$ and a fixed index $i \in [n]$, any feasible
		solution of~(\ref{ILP 3 3}) satisfies $u_i = 1$. Indeed, if $u_i = 0$,
		then $\lambda_i = 0$, and the stationarity condition
  implies $\nu_i = 2$, contradicting
		$\nu_i \leq M^d v_i \leq 1$. Next, it follows from (\ref{cons: minmax big M}) that $y_i = \tfrac{x_i}{2}$, and $x_i \in [0, \tfrac{2}{3}]$ for each $i \in [n]$. Conversely, for every
		$x_i \in [0,\tfrac{2}{3}]$, the~corresponding constraints in
		(\ref{cons: minmax big M}) are satisfied by choosing
		$y_i=\tfrac{x_i}{2}$, $u_i=1$, $v_i=0$, $\lambda_i=1$, and~$\nu_i=0$.
		
	 Consequently, projecting the feasible region of (\ref{ILP 3 3}) onto the leader variables $\mathbf{x}$ yields
		\begin{subequations} \label{ILP 4 3}
			\begin{align}
				\min_{\mathbf{x}} \;
				& \Big\{\eta \mathbf{c}^\top \mathbf{x}
				+ \sum_{i = 1}^{n} x_i\Big\} \label{obj: reduced minmax MILP}\\
				\mbox{s.t. } 
				& \mathbf{A}\mathbf{x} \leq \mathbf{b} \label{cons: reduced minmax leader}\\
				& \mathbf{0} \leq \mathbf{x} \leq \frac{2}{3}\mathbf{1}. \label{cons: reduced minmax bounds}
			\end{align}
		\end{subequations}
		Then, for any feasible $\mathbf{x} \neq \mathbf{0}$, Hölder's inequality implies 
		\[
		\eta \mathbf{c}^\top \mathbf{x} + \sum_{i = 1}^n x_i
		\geq
		- \eta \Vert \mathbf{c} \Vert_\infty \Vert \mathbf{x} \Vert_1
		+ \Vert \mathbf{x} \Vert_1
		=
		\bigl(1 - \eta \Vert \mathbf{c} \Vert_\infty\bigr)
		\Vert \mathbf{x} \Vert_1
		> 0,
		\]
		where the last inequality follows from
		$\eta \Vert \mathbf{c} \Vert_\infty
		\leq \eta \Vert \mathbf{c} \Vert_1
		\leq \tfrac{1}{2D} < 1$.
		
	  As a result, any optimal solution of the big-$M$ formulation (\ref{ILP 3 3}) satisfies $\tilde{\mathbf{x}}^* = \mathbf{0}$ and $\tilde{\mathbf{y}}^* = \mathbf{0}$. We conclude that with $\eta = \tfrac{1}{2D(\Vert \mathbf{c} \Vert_1 + 1)}$ a ``yes''-instance of [\textbf{ILP-D}] implies a ``yes''-instance of~[\textbf{BLP-D}], and vice versa.  This observation implies the result.
	\end{proof}
\end{theorem}
\color{black}

We next consider a big-$M$ MILP formulation tailored to min-max instances of [\textbf{BLP}$'$]. In this case, the follower's problem in~(\ref{cons: follower binary}) can be directly replaced by its dual problem and the primal follower variables~$\mathbf{y}$ are no longer necessary. We therefore obtain the following big-$M$ MILP formulation of~[\textbf{BLP}$'$] with $\mathbf{G} = \mathbf{0}$ and $\mathbf{d} = -\mathbf{g}$:
\begin{subequations}
	\label{MILP MM}
	\begin{align}
		\min_{\mathbf{x}, \boldsymbol{\lambda}, \mathbf{u}} \; &
		\Big\{ \mathbf{a}^{\top}\mathbf{x}
		- \sum_{i = 1}^{n_l} \sum_{j = 1}^{m_f} L_{ji} u_{ij}
		+ \mathbf{f}^{\top}\boldsymbol{\lambda} \Big\} \\
		\text{s.t. } & \mathbf{x} \in \{0,1\}^{n_l} \\
		& \mathbf{H}\mathbf{x} \leq \mathbf{h} \\
		& \mathbf{g} + \mathbf{F}^{\top}\boldsymbol{\lambda} = \mathbf{0} \\
		& \boldsymbol{\lambda} \geq \mathbf{0} \\	
		& \text{(\ref{cons: stong duality linear 1})--(\ref{cons: stong duality linear 3})}
		\quad \forall i \in [n_l], \quad \forall j \in [m_f].
	\end{align}
\end{subequations}

To ensure the validity of formulation~(\ref{MILP MM}),
we make the following additional assumption:
\begin{itemize}
	\item[\textbf{A1$'$.}] The set
	$X_b := \{\mathbf{x} \in \{0, 1\}^{n_l} :
	\mathbf{H}\mathbf{x} \leq \mathbf{h}\}$ is non-empty,
	and the follower's problem in~(\ref{cons: follower binary})
	admits a finite optimal solution for every~$\mathbf{x} \in X_b$.
\end{itemize}
We then define the following optimality verification problem:
\begin{itemize}
	\item[$ $] [\textbf{MM-D}$'$]: Given a min-max instance of [\textbf{BLP}$'$] with rational data, rational parameters 
	$\mathbf{M} \nobreak\in\nobreak \mathbb{Q}^{m_f}_{>0}$, and $\tilde{\mathbf{x}}^*$ corresponding to an optimal solution of (\ref{MILP MM}), decide whether $\tilde{\mathbf{x}}^*$ is optimal for~[\textbf{BLP}$'$].
\end{itemize}
The following result establishes that optimality verification remains computationally difficult for this specialized reformulation. \color{black}

\begin{theorem} \label{theorem 4}
	The optimality verification problem  \upshape [\textbf{MM-D}$'$] \itshape \color{black} is $\mathrm{coNP}$-complete even when $M_j = M$ for all $j \in [m_f]$, $\mathbf{G} = \mathbf{0}$ and $\mathbf{d} = -\mathbf{g}$. 
	\begin{proof}  Membership in $\mathrm{coNP}$ follows analogously to Theorem \ref{theorem 1}. \color{black} Similar to the proof of Theorem \ref{theorem 2}, we use a polynomial-time reduction from [\textbf{ILP-D}].  To this end, \color{black} note that (\ref{ILP}) can be expressed as a min-max problem of the form:
		\begin{subequations} \label{ILP 1 4}
			\begin{align}
				\min_{\mathbf{x}, \mathbf{y}^*} & \; \Big\{ \mathbf{c}^\top \mathbf{x} + \sum_{i = 1}^{n} 2 y^*_i \Big\}\\
				\mbox{s.t. } & \mathbf{A}\mathbf{x} \leq \mathbf{b} \\
				& \mathbf{x} \in \{0, 1\}^n \\
				& \mathbf{y}^* \in \argmin_{\,\mathbf{y}}
				\Big\{ -\sum_{i = 1}^{n} 2 y_i \Big\} \\
				& \text{s.t.} \quad 2\mathbf{y} \leq \eta \mathbf{x} \label{cons: scaled follower 1} \\
				& \phantom{\text{s.t.} \quad} \mathbf{y} \leq \eta (\mathbf{1} - \mathbf{x}), \label{cons: scaled follower 2}	
			\end{align}
		\end{subequations}
	  where $\eta \in \mathbb{Q}_{>0}$ is a  scaling \color{black} parameter specified later.  By construction, (\ref{ILP 1 4}) constitutes a min-max instance of [\textbf{BLP}$'$] satisfying Assumption~\textbf{A1$'$}. \color{black} Furthermore,  $\mathbf{x} \in \{0,1\}^n$ implies that  \[y^*_i = \min\{\frac{\eta}{2} x_i, \eta(1 - x_i)\} = 0 \quad  \forall i \in [n],\] \color{black} and therefore $\sum_{i=1}^{n} 2y_i^* = 0$.

	 Let $\boldsymbol{\lambda} \in \mathbb{R}_+^n$ and $\boldsymbol{\nu} \in \mathbb{R}_+^n$ be the dual variables corresponding to the follower constraints~(\ref{cons: scaled follower 1})--(\ref{cons: scaled follower 2}), respectively. Similar to the proof of Theorem \ref{theorem 2}, \color{black} we obtain the following dual 
	 formulation \color{black} of (\ref{ILP 1 4}):
	  \begin{subequations} \label{ILP 2 4}
	  	\begin{align}
	  		\min_{\mathbf{x}, \boldsymbol{\lambda}, \boldsymbol{\nu}, \mathbf{u}, \mathbf{v}} \,
	  		&  \Big\{\mathbf{c}^\top \mathbf{x} 
	  		+ \eta \sum_{i = 1}^{n} \big( u_i + \nu_i - v_i \big)\Big\} \\
	  		\mbox{s.t. } 
	  		& \mathbf{x} \in \{0,1\}^n	\\
	  		& \mathbf{A}\mathbf{x} \leq \mathbf{b} \\
	  		& \left.
	  		\begin{aligned}
	  			& 2\lambda_i + \nu_i = 2 \\
	  			& 0 \leq u_i \leq \lambda_i \\
	  			& 0 \leq v_i \leq \nu_i \\
	  			& u_i \leq M_i x_i \\
	  			& v_i \leq \tilde{M}_i x_i \\
	  			& u_i \geq \lambda_i + M_i (x_i - 1) \quad \\
	  			& v_i \geq \nu_i + \tilde{M}_i (x_i - 1)
	  		\end{aligned}
	  		\right\}
	  		\quad \forall i \in [n].
	  	\end{align}
	  \end{subequations}	
		
	 Next, \color{black} by setting $M_i = \tilde{M}_i = 1$, $i \in [n]$, we observe that $\mathbf{x} = \mathbf{u} = \mathbf{v} = \boldsymbol{\nu} = \mathbf{0}$ and $\boldsymbol{\lambda} = \mathbf{1}$ is a feasible solution of (\ref{ILP 2 4}) with the objective value $0$. In contrast, $x_j = 1$ for some $j \in [n]$ implies that $u_j = \lambda_j$, $v_j = \nu_j$ and $\lambda_j \in [\tfrac{1}{2}, 1]$. Then, $u_i \geq 0$ and $\nu_i - v_i \geq 0$ for each $i \in [n]$~yields
	\[\sum_{i = 1}^{n} \big( u_i + \nu_i - v_i \big) \geq \lambda_j \geq \tfrac{1}{2}.\] 
	We conclude that, for $\eta > 2\Vert \mathbf{c} \Vert_{1}$ and any binary $\mathbf{x} \neq \mathbf{0}$,
	\[\mathbf{c}^\top \mathbf{x} + \eta \sum_{i = 1}^{n} \big( u_i + \nu_i - v_i \big) > - \Vert \mathbf{c} \Vert_{1} \Vert\mathbf{x}\Vert_{\infty} + \Vert \mathbf{c} \Vert_{1} = 0.\]

	As a result, any leader-optimal solution of~(\ref{ILP 2 4}) with $\eta := 2 \Vert \mathbf{c} \Vert_1 + 1$ satisfies $\tilde{\mathbf{x}}^* = \mathbf{0}$. 
	On the other hand, $\tilde{\mathbf{x}}^* = \mathbf{0}$ is optimal 
	for~(\ref{ILP 1 4}) if and only if it is optimal for~(\ref{ILP}). 
	Hence, [\textbf{ILP-D}] is a ``yes''-instance if and only if 
	[\textbf{MM-D}$'$] is a ``yes''-instance, and the result follows. \color{black}
	\end{proof}	
\end{theorem}

Unlike Corollaries~\ref{corollary 1} and~\ref{corollary 2}, the approximation result does not directly extend to Theorem~\ref{theorem 3}, since the objective function in (\ref{eq: ILP concave}) may attain its maximum at a fractional $\mathbf{x}$. On a positive note, 
a similar approximation result can be established for the construction in Theorem~\ref{theorem 4}.

Let $z_b^{\prime*}$ and $z^{\prime \, \text{max}}_b$ denote the 
optimal and the maximum feasible objective function values of a min-max instance 
of [\textbf{BLP}$'$], respectively, whenever both are finite. Then, an $\varepsilon$-approximate optimality verification problem for~[\textbf{MM-D}$'$] is defined as:
\begin{itemize}
	\item[$ $] [\textbf{MM-D}$'_{\varepsilon}$]: Given the input of 
	[\textbf{MM-D}$'$] and a rational parameter $\varepsilon \in (0,1)$, decide whether
	$\mathbf{a}^\top \tilde{\mathbf{x}}^* + \mathbf{d}^\top \tilde{\mathbf{y}}^* - z_b^{\prime *}
	\leq
	\varepsilon(z_b^{\prime \max} - z_b^{\prime *})$.
\end{itemize}
Here, $\tilde{\mathbf{y}}^*$ is an optimal solution of the follower's problem in (\ref{cons: follower binary}) corresponding to $\mathbf{x} = \tilde{\mathbf{x}}^*$. Similar to Corollaries~\ref{corollary 1} and~\ref{corollary 2}, the following result holds. 
\begin{corollary} \label{corollary 3}
	Under the conditions of Theorem~\ref{theorem 4}, \upshape
	[\textbf{MM-D}$'_{\varepsilon}$] \itshape is $\mathrm{coNP}$-hard for every fixed  rational~$\varepsilon\in\nobreak(0, 1)$.
\end{corollary}
\color{black}

\section{Global big-$M$  verification \color{black}} \label{sec: global}

Kleinert~et~al.~\cite{Kleinert2020} establish that verifying whether the dual big-$M$ parameter $M^d$ cuts off feasible or optimal vertices of the follower's dual polyhedron \color{black} is computationally difficult when performed \textit{a priori}, i.e., before solving~(\ref{MILP}).
We first show that the corresponding verification
problems remain difficult \color{black} even \textit{a posteriori}, i.e., when an optimal solution to~(\ref{MILP}) is
available. 

\subsection{Valid bounds for dual feasible solutions}	
The follower's dual polyhedron is defined by constraints (\ref{cons: dual feasibility 1})-(\ref{cons: dual feasibility 2}),  i.e., \color{black}
\begin{equation} \label{dual polyhedron}
	P := \big\{\boldsymbol{\lambda} \in \mathbb{R}^{m_f}_+: \; \mathbf{F}^\top \boldsymbol{\lambda} + \mathbf{g} = \mathbf{0} \big\}.
\end{equation}
 Analogously to \cite{Kleinert2020}, \color{black} a decision version of the global big-$M$ verification problem~(GVP) is formulated as follows: 
\begin{itemize}
	\item[$ $] [\textbf{GVP-D}]: Given  an instance of [\textbf{BLP}] with rational data, \color{black} an index $j \in [m_f]$, and an optimal solution of the MILP  formulation \color{black} (\ref{MILP}) with $M^d,  M^p \color{black} \in \mathbb{Q}_{>0}\color{black}$, decide whether every vertex $\mathbf{v}\in\nobreak\text{vert}(P)$ satisfies~$v_j \leq M^d$.
\end{itemize}

In the following, we use a reduction from the problem of finding a component-wise optimal vertex of a polyhedron:
\begin{itemize}
	\item[$ $] [\textbf{COVP-D}]: Given a  rational \color{black} polyhedron
	\begin{equation} \label{polyhedron 2}
		\tilde{P} := \big\{\boldsymbol{\lambda} \in \mathbb{R}_+^{m}: \; \tilde{\mathbf{F}}^\top \boldsymbol{\lambda} + \tilde{\mathbf{g}} = \mathbf{0} \big\},
	\end{equation}
a bound $\tilde{M} \in \mathbb{Q}_{>0}\color{black}$ such that $\tilde{P} \cap [0, \tilde{M}]^m \neq \emptyset$, and an index $j \in [m]$, decide whether every vertex $\mathbf{v} \in \operatorname{vert}(\tilde{P})$ satisfies $v_j \le \tilde{M}$.
\end{itemize}

While Kleinert et al.~\cite{Kleinert2020} establish $\mathrm{coNP}$-completeness of  the corresponding verification problem for polyhedra defined by general linear inequalities, \color{black} we consider  [\textbf{COVP-D}] for \color{black} polyhedra of the form~(\ref{polyhedron 2}) with 
the additional property that $\tilde{P} \cap [0, \tilde{M}]^m \neq \emptyset$. To establish $\mathrm{coNP}$-completeness in this restricted setting, we rely on the Hamiltonian-path construction of~\cite{Fukuda1997}.

\begin{lemma} \label{lemma 2}
The decision problem \upshape [\textbf{COVP-D}] \itshape is $\mathrm{coNP}$-complete. 
\begin{proof}
 
  First, [\textbf{COVP-D}] belongs to $\mathrm{coNP}$. Indeed, for a ``no'' instance, a certificate is a vertex $\mathbf{v} \in \operatorname{vert}(\tilde{P})$ satisfying $v_j > \tilde{M}$. Such a vertex has polynomial encoding length, and its feasibility and vertex property can be verified in polynomial time.
  \color{black}
  
	 To establish $\mathrm{coNP}$-hardness, \color{black} we slightly modify the reduction from the directed Hamiltonian path problem used in Theorem~4.1 of \cite{Fukuda1997}. Thus, given a directed graph $G=(V,A)$ with source $s$ and destination $t$, let $x_{ij}$ denote the flow on arc $(i,j) \in A$. The $s$--$t$ path polyhedron is then defined as:
\begin{equation} \nonumber
	P_{st\color{black}}:=\Big\{\mathbf{x} \in\mathbb{R}^{|A|}_+:\ 
	\sum_{j : \, (v,j)\in A} x_{vj} - \sum_{i: \, (i,v)\in A} x_{iv}
	=
	\begin{cases}
		\phantom{-}1, & v=s,\\
		-1, & v=t,\\
		\phantom{-}0, & v\in V\setminus\{s,t\},
	\end{cases}
	\ \forall v\in V
	\Big\}.
\end{equation}
	Without loss of generality, assume that $|V| \geq 3$ and a directed arc $(s, t) \in A$ is present. Furthermore, we set the cost vector $\mathbf{c} = \mathbf{1}$ and choose $\tilde{M} = |V| - 2$. 
	
	By introducing  an auxiliary \color{black} variable $\tau$, we consider the lifted polyhedron
	\[
	\tilde P \;:=\;\{(\mathbf{x}, \tau)\in\mathbb{R}_+^{|A|+1}:\ \mathbf{x} \in P_{st\color{black}},\ \tau- \mathbf{c}^\top \mathbf{x}=0\},
	\]
	which can be written in the form \eqref{polyhedron 2}. Since $|V| \geq 3$ and $(s, t) \in A$, the trivial path $(s,t)$ yields a feasible point of $\tilde{P}$ with $\tau = 1 \leq \tilde{M}$. Therefore, $\tilde{P} \cap [0, \tilde{M}]^{|A| + 1} \neq \emptyset$, and we thus obtain an instance of~[\textbf{COVP-D}] with index~$j$ corresponding to~$\tau$.
	 
	Next, since $\tau$ is uniquely determined by $\mathbf{x}$, the vertices of $\tilde{P}$ have the form~$(\mathbf{v}, \mathbf{c}^\top \mathbf{v})$ where $\mathbf{v} \in\nobreak\text{vert}(P_{st\color{black}})$. Furthermore, Theorem~4.1 in \cite{Fukuda1997} shows that the vertices of~$P_{st\color{black}}$ are exactly the characteristic vectors of simple $s$--$t$ paths in $G$. As a result, $G$ has no Hamiltonian~$s$--$t$ path if and only if every vertex $(\mathbf{x}, \tau) \in \text{vert}(\tilde{P})$ satisfies $\tau \leq \tilde{M} =|V|- 2$, i.e., the answer to [\textbf{COVP-D}] is ``yes''.
\end{proof}
\end{lemma}

\begin{theorem} \label{theorem 5}
The decision problem \upshape [\textbf{GVP-D}] \itshape is $\mathrm{coNP}$-complete. 
\begin{proof}
 
First, analogously to the proof of Lemma~\ref{lemma 2}, [\textbf{GVP-D}] belongs to $\mathrm{coNP}$. To establish $\mathrm{coNP}$-hardness, \color{black} given an instance of [\textbf{COVP-D}], we consider the following instance of [\textbf{BLP}]:
\begin{subequations}
	\label{BLP global}
	\begin{align}
		\min_{\mathbf{x}, \mathbf{y}^*} \; & 
		0 \\
		\text{s.t. } 
		& \mathbf{x} = \mathbf{0} \\
		& \mathbf{y}^* \in 
		\argmin_{\;\mathbf{y}} \,
		\Big\{\tilde{\mathbf{g}}^{\top}\mathbf{y}: \tilde{\mathbf{F}} \mathbf{y} \leq \mathbf{0} \Big\}. \label{cons: BLP global follower}
	\end{align}
\end{subequations}
Notably, $\mathbf{y} = \mathbf{0}$ is feasible for the follower's problem in (\ref{cons: BLP global follower}). Moreover, the follower's dual feasible set coincides with $\tilde{P}$, which is non-empty by design. Hence, the follower's problem in~(\ref{cons: BLP global follower}) admits a finite optimum, and Assumption~\textbf{A1} holds. 

In addition, the  big-$M$ MILP  formulation \color{black} of (\ref{BLP global}) is given~by:
\begin{subequations}
	\label{MILP global}
	\begin{align}
		\min_{\mathbf{x}, \mathbf{y}, \boldsymbol{\lambda}, \mathbf{u}} \; & 
		0 \\
		\text{s.t. } 
		& \mathbf{x} = \mathbf{0} \\
		& \tilde{\mathbf{F}}^\top \boldsymbol{\lambda} + \tilde{\mathbf{g}} = \mathbf{0} \\
		& \mathbf{0} \leq \boldsymbol{\lambda} \leq M^d \mathbf{u} \\
		& \mathbf{0} \leq -\tilde{\mathbf{F}} \mathbf{y} \leq M^p (\mathbf{1} - \mathbf{u}) \\
		& \mathbf{u} \in \{0, 1\}^m.
	\end{align}
\end{subequations}
 Since the objective function of (\ref{MILP global}) is zero, any feasible solution of (\ref{MILP global}) is also optimal. 
In particular, we choose \color{black} $\tilde{\mathbf{x}}^* = \tilde{\mathbf{y}}^* = \mathbf{0}$, $\tilde{\mathbf{u}}^* = \mathbf{1}$, and any feasible $\tilde{\boldsymbol{\lambda}}^* \in \tilde{P}$ with $\tilde{\boldsymbol{\lambda}}^* \leq M^d \mathbf{1}$.  With $M^d = \tilde{M}$ and~$M^p = 1$, \color{black} the outlined optimal solution exists  since \color{black} $\tilde{P} \cap [0, \tilde{M}]^m \neq \emptyset$.  Moreover, \color{black} it can be computed in polynomial time  by solving the associated LP feasibility problem. \color{black} 

 Since the MILP formulation (\ref{MILP global}) does not depend on the answer to [\textbf{COVP-D}], the constructed solution  \color{black} remains optimal  for both ``yes'' and ``no'' instances of [\textbf{COVP-D}]. \color{black} By construction,~the answer to [\textbf{GVP-D}] is ``yes'' if and only if  the original instance of \color{black} [\textbf{COVP-D}] is a ``yes''~instance,  and the result~follows. \color{black}
\end{proof}
\end{theorem}
\subsection{Valid bounds for bilevel feasible solutions}	
 Similar to \color{black} Kleinert~et~al.~\cite{Kleinert2020}, we define a global big-$M$ verification problem over the  projected \color{black} leader's feasible region (\ref{eq: leader's feasible set extended}) as follows:
\begin{itemize}
	\item[$ $] [\textbf{GVP-D$'$}]: Given  an instance of [\textbf{BLP}] with rational data \color{black} and
	an optimal solution of the MILP  formulation \color{black} (\ref{MILP}) with $ M^d, M^p \in \mathbb{Q}_{>0} \color{black}$, decide whether for every feasible $\mathbf{x} \in H_x$ there exists  an optimal dual solution \color{black}
	\begin{equation} \nonumber
	 	\boldsymbol{\lambda}^* \color{black} \in 
		\argmax_{\,\boldsymbol{\lambda}\color{black}} \Big\{
		(\mathbf{L}\mathbf{x} - \mathbf{f})^\top \boldsymbol{\lambda}: \; \boldsymbol{\lambda} \in P \color{black} \Big\}
	\end{equation}
	such that $\lambda^*_j \leq M^d$ for all $j \in [m_f]$.
\end{itemize}

To prove $\mathrm{coNP}$-completeness, we use a reduction from the complement of the partition problem~(PP) given by: 
\begin{itemize}
	\item[$ $] [\textbf{PP-D}]: Given a set of positive integers $a_i \in \mathbb{Z}_{>0}$, $i \in [n]$, decide whether $\mathbf{a}^\top \boldsymbol{\sigma} \neq 0$ for all $\boldsymbol{\sigma} \in \{-1, 1\}^n$.
\end{itemize}
Since the complement of [\textbf{PP-D}] is $\mathrm{NP}$-complete \cite{Garey1979}, the problem is $\mathrm{coNP}$-complete. Furthermore, the following results hold.

\begin{lemma} \label{lemma 3}
	If the answer to \upshape [\textbf{PP-D}] \itshape is ``yes'', then, for all 
	$\mathbf{x} \in [-1,1]^n$ such that 
	$\mathbf{a}^\top \mathbf{x} = 0$, we have $\Vert \mathbf{x} \Vert_1 \leq n - 1/\Vert\mathbf{a}\Vert_\infty$.
	\begin{proof} 
	Let $\boldsymbol{\sigma} := \text{sign}(\mathbf{x})$, where we adopt the convention that if $x_i = 0$, then $\sigma_i \in \{-1, 1\}$ can be chosen arbitrarily. Since $\mathbf{a}^\top \boldsymbol{\sigma}$ is integer and the answer to [\textbf{PP-D}] is ``yes'', we conclude that~$|\mathbf{a}^\top \boldsymbol{\sigma}| \geq 1$. Furthermore, using $\mathbf{a}^\top \mathbf{x} = 0$ and applying Hölder's inequality yields
	\[1 \leq |\mathbf{a}^\top \boldsymbol{\sigma}| = |\mathbf{a}^\top (\boldsymbol{\sigma} - \mathbf{x})| \leq \Vert\mathbf{a}\Vert_\infty \Vert \boldsymbol{\sigma} - \mathbf{x} \Vert_1. \]
	Finally, we note that $\Vert \boldsymbol{\sigma} - \mathbf{x} \Vert_1 = \sum_{i = 1}^n |\sigma_i - x_i| = n - \Vert \mathbf{x} \Vert_1$, which implies the result.
	\end{proof}
\end{lemma}
	
	\begin{theorem} \label{theorem 6}
		The decision problem \upshape [\textbf{GVP-D}$'$] \itshape is $\mathrm{coNP}$-complete. 
		\begin{proof}
		
		For a fixed
		$\mathbf{x} \in H_x$, let
	 \begin{align*}
	 &	\varphi(\mathbf{x})
		:=
		\max_{\boldsymbol{\lambda}}
		\left\{
		(L\mathbf{x}-\mathbf{f})^\top \boldsymbol{\lambda}
		:
		\boldsymbol{\lambda} \in P
		\right\}, \\
	&	\varphi_M(\mathbf{x})
		:=
		\max_{\boldsymbol{\lambda}}
		\left\{
		(L\mathbf{x}-\mathbf{f})^\top \boldsymbol{\lambda}
		:
		\boldsymbol{\lambda} \in P,\ 
		\boldsymbol{\lambda} \le M^d \mathbf{1}
		\right\}.
	\end{align*}
	  Then, [\textbf{GVP-D}$'$] is a ``no'' instance if and
	  only if there exists $\hat{\mathbf{x}} \in H_x$ such that
	  $\varphi(\hat{\mathbf{x}}) > \varphi_M(\hat{\mathbf{x}})$.
	  A~polynomial-size certificate consists of optimal bases of the two corresponding linear programs at~$\hat{\mathbf{x}}$. For fixed bases, the conditions that they are optimal at
	  $\mathbf{x}$, as well as the strict inequality $\varphi(\mathbf{x}) > \varphi_M(\mathbf{x})$ are
	  linear in $\mathbf{x}$. Therefore, the existence of
	  $\mathbf{x} \in H_x$ satisfying these conditions can be checked
	  by solving a linear feasibility problem, and
	  [\textbf{GVP-D}$'$] belongs to~$\mathrm{coNP}$.

			
		To establish $\mathrm{coNP}$-hardness, \color{black} given an instance of [\textbf{PP-D}], we consider the following instance of~[\textbf{BLP}]:	
		\begin{subequations}
			\label{BLP global 2}
			\begin{align}
				\min_{\mathbf{x}, \mathbf{y}^*, z^*} \; & 
				0 \\
				\text{s.t. } 
				& \mathbf{x} \in [-1,1]^n \label{cons: leader global 1} \\
				& \mathbf{a}^\top \mathbf{x} = 0 \label{cons: leader global 2} \\
				& (\mathbf{y}^*, z^*) \in 
				\argmin_{\;\mathbf{y}, z} \,
				2z \label{obj: follower global} \\ 
				& \text{\quad \quad s.t.} \quad y_i \geq x_i \quad \forall i \in [n] \label{cons: follower global 1}\\
				& \phantom{\quad \quad s.t. \quad} y_i \geq -x_i \quad \forall i \in [n] \label{cons: follower global 2} \\
				& \phantom{\quad \quad \text{s.t.} \quad} z \geq \sum_{i = 1}^n y_i - K \label{cons: follower global 3}\\
				& \phantom{\quad \quad \text{s.t.} \quad} 2z \geq 0, \label{cons: follower global 4} 	
			\end{align}
		\end{subequations}
		where $K = n - 1/\Vert\mathbf{a}\Vert_\infty \geq 0$. It can be verified that Assumption \textbf{A1} holds and, furthermore,  the follower's optimal objective function value \color{black} in (\ref{obj: follower global})--(\ref{cons: follower global 4})  is given by
		$2z^* = 2\max\{\Vert \mathbf{x} \Vert_1 - K; 0\}$.\color{black}
		
		Let $\boldsymbol{\lambda} \in \mathbb{R}_+^{n}$, $\boldsymbol{\nu} \in \mathbb{R}_+^{n}$, $\theta \in \mathbb{R}_+$ and $\psi \in \mathbb{R}_+$ be dual variables corresponding to the follower constraints~(\ref{cons: follower global 1})--(\ref{cons: follower global 4}), respectively. Then, the follower's dual problem is given by:
		\begin{subequations}
			\label{global dual problem}
			\begin{align}
				\max_{\boldsymbol{\lambda}, \boldsymbol{\nu}, \theta, \psi} \; & 
				\Big\{(\boldsymbol{\lambda} - \boldsymbol{\nu})^\top \mathbf{x} - K \theta \Big\} \\
				\text{s.t. } 
				& \boldsymbol{\lambda} + \boldsymbol{\nu} = \theta \mathbf{1} \label{cons: follower daul global 1} \\
				& \theta + 2\psi = 2 \label{cons: follower daul global 2} \\
			  & \boldsymbol{\lambda} \geq \mathbf{0}, \; \boldsymbol{\nu} \geq \mathbf{0}, \; \theta \geq 0, \; \psi \geq 0. \label{cons: follower daul global 3}
			\end{align}
		\end{subequations}
			Furthermore, the  big-$M$ MILP formulation \color{black} of (\ref{BLP global 2}) reads as: 
		  \begin{subequations}
		  	\label{MILP global 2}
		  	\begin{align}
		  		\min_{\mathbf{x}, \mathbf{y}, z, \boldsymbol{\lambda}, \boldsymbol{\nu}, \theta, \psi, \mathbf{u}, \mathbf{v}, p, q} \; & 
		  		0 \\
		  		\text{s.t. } & \text{(\ref{cons: leader global 1})--(\ref{cons: leader global 2}), (\ref{cons: follower global 1})--(\ref{cons: follower global 4}), (\ref{cons: follower daul global 1})--(\ref{cons: follower daul global 3})}, \\ 
		  		& \mathbf{0} \leq \boldsymbol{\lambda} \leq M^d \mathbf{u} \\
		  		& \mathbf{0} \leq \mathbf{y} - \mathbf{x} \leq M^p (\mathbf{1} - \mathbf{u}) \\
		  		& \mathbf{0} \leq \boldsymbol{\nu} \leq M^d \mathbf{v} \\
		  		& \mathbf{0} \leq \mathbf{x} + \mathbf{y} \leq M^p (\mathbf{1} - \mathbf{v}) \\
		  		& 0 \leq \theta \leq M^d p\color{black} \\
		  		& 0 \leq z - \sum_{i = 1}^n y_i + K \leq M^p (1 - p\color{black}) \\
		  		& 0 \leq \psi \leq M^d q \\
		  		& 0 \leq z \leq \tfrac{1}{2} M^p (1 - q) \\
		  		& \mathbf{u} \in \{0, 1\}^n, \; \mathbf{v} \in \{0, 1\}^n, \; p\color{black} \in \{0, 1\}, \; q \in \{0, 1\}.
		  	\end{align}
		  \end{subequations}
		  
		Assume that $M^d = 1$ and $M^p = 2n\color{black}$.  Since the objective function in (\ref{MILP global 2}) is zero, any feasible solution is optimal. We construct such a solution by setting \color{black} all variables to zero, except $\tilde{\psi}^* = \tilde{q}^* = 1$.
		 Notably,  the constructed \color{black} solution remains optimal regardless of whether the underlying instance of~[\textbf{PP-D}] is a ``yes'' or a~``no''~instance.
		
		Let $P$ denote the dual feasible set defined by the constraints (\ref{cons: follower daul global 1})--(\ref{cons: follower daul global 3}). First, we assume that the answer to [\textbf{PP-D}] is ``yes''. In this case, Lemma \ref{lemma 3} implies that for all $\mathbf{x}$ satisfying the leader constraints~(\ref{cons: leader global 1})--(\ref{cons: leader global 2}), we have $\Vert \mathbf{x} \Vert_1 \leq K$. Hence, $2z^* = 0$, and strong duality implies that the optimal objective function value in~(\ref{global dual problem}) is also equal to zero. Furthermore, it can readily be checked that $\boldsymbol{\lambda} = \boldsymbol{\nu} = \mathbf{0}$, $\theta = 0$ and $\psi = 1$  is an optimal dual solution. \color{black} 
		We therefore conclude that [\textbf{GVP-D$'$}] with~$M^d = 1$  and $M^p = 2n\color{black}$ \color{black} admits a ``yes''~instance.  
		
		Now, assume that the answer to [\textbf{GVP-D}$'$] is ``yes''. Suppose, for contradiction, that [\textbf{PP-D}] is a ``no'' instance. That is, there exists $\tilde{\mathbf{x}} \in \{-1, 1\}^n$ such that $\mathbf{a}^\top \tilde{\mathbf{x}} = 0$. In particular, $\Vert \tilde{\mathbf{x}} \Vert_1 = n$ and the respective follower's optimal objective function value in (\ref{obj: follower global})--(\ref{cons: follower global 4}) is given by $2z^* = 2(n - K) > 0$.
		Next, we note that the optimal dual objective function value in (\ref{global dual problem}) at $\mathbf{x} = \tilde{\mathbf{x}}$ is bounded from above~as:
		\[(\boldsymbol{\lambda} - \boldsymbol{\nu})^\top \tilde{\mathbf{x}} - K \theta \leq \sum_{i = 1}^n |\lambda_i - \nu_i| - K \theta \leq \sum_{i = 1}^n (\lambda_i + \nu_i) - K\theta = (n - K)\theta.\]
		Hence,  any optimal dual solution \color{black} attaining the objective value $2(n-K)$ must satisfy $\theta = 2$. Since $\theta = 2 > M^d$,  no optimal dual solution at $\mathbf{x} = \tilde{\mathbf{x}}$ can satisfy the dual bound with $M^d = 1$. \color{black} This contradicts the assumption that [\textbf{GVP-D}$'$] is a ``yes'' 
		instance. Therefore, [\textbf{PP-D}] admits a ``yes'' instance, and the 
		result follows.
	  \end{proof}  
	\end{theorem}  
  
  Importantly, \color{black} the \textit{a posteriori} hardness result of Theorem~\ref{theorem 6} differs from the corresponding \textit{a priori} result of Kleinert et al.~\cite{Kleinert2020} in  two \color{black} important aspects. First, Theorem~\ref{theorem 6} is established assuming that an optimal 
  solution of the MILP  formulation \color{black} (\ref{MILP}) is given as input.  Since our reduction constructs such an optimal solution in polynomial time, \color{black} it immediately implies the corresponding \textit{a priori} hardness result when this input is omitted. 
  Second, while the complexity result in \cite{Kleinert2020} is established via a connection to bilevel feasibility, Theorem~\ref{theorem 6} is based on an explicit polynomial-time reduction from the complement of the partition problem.

 In addition, the reduction in Theorem~\ref{theorem 6} yields a stronger
 hardness result: even verifying whether the pair $(M^p, M^d)$ preserves all bilevel-optimal solutions of the original bilevel problem is $\mathrm{coNP}$-hard.
  \begin{corollary} \label{corollary 4}
  	Verifying the bilevel correctness of the pair
  	$(M^p, M^d)$ for the MILP formulation
  	\upshape (\ref{MILP}) \itshape of~\upshape[\textbf{BLP}] \itshape is
  	$\mathrm{coNP}$-hard, even when an optimal
  	solution of \upshape (\ref{MILP}) \itshape is provided as input.
  \begin{proof} The result follows from the reduction in
  	Theorem~\ref{theorem 6}.
  	Indeed, in the constructed instances, every bilevel-feasible solution
  	$(\mathbf{x}, \mathbf{y}^*, z^*)$ satisfies
  	$\mathbf{x} \in [-1, 1]^n$ and
  	\[
  	z^* = \max\{\|\mathbf{x}\|_1 - K, 0\} \leq n - K.
  	\]
  	Hence, $\sum_{i = 1}^n y_i^* \leq z^* + K \leq n$,
  	implying $y_i^* \leq n$ for each $i \in [n]$.
  	Thus, all primal slacks in the big-$M$ formulation
  	(\ref{MILP global 2}) are bounded by $2n$.
  	Moreover, every bilevel-feasible solution of
  	(\ref{BLP global 2}) is also bilevel-optimal due to
  	the zero objective. As a result, the pair
  	$(M^p, M^d) = (2n, 1)$ is bilevel correct, i.e.,
  	preserves all bilevel-optimal solutions for the constructed instance of
  	[\textbf{BLP}], if and only if [\textbf{PP-D}]
  	is a ``yes'' instance.
  \end{proof}
 \end{corollary}
 

 \section{Conclusion} \label{sec: conclusion}
  In this paper, we analyze mixed-integer linear programming (MILP) reformulations of bilevel linear programs (BLPs)  and mixed-integer BLPs \color{black} that involve big-$M$ parameters. Specifically, we ask whether an optimal solution of the big-$M$  MILP formulation \color{black} is \textit{bilevel optimal} and whether the given big-$M$ parameters are \textit{bilevel correct}, i.e.,  preserve all bilevel-optimal solutions of the original bilevel problem. \color{black}
 While previous results in the literature have primarily focused on the \textit{a priori} validation of the big-$M$ parameters, we consider an \textit{a posteriori} problem, where an optimal solution of the big-$M$  formulation \color{black} is provided as~input. 
 
  An efficient algorithm for either of these verification problems could potentially be used to tune the big-$M$ parameters in practice. However, we establish that even in a very restricted problem setting~both problems remain computationally difficult.  Moreover, our results show that repeatedly increasing the big-$M$ parameters and resolving the resulting big-$M$ formulation does not, in general, provide a tractable verification procedure.
  
  \color{black} Taken together, these negative results highlight inherent limitations of big-$M$ MILP formulations for bilevel optimization problems. In particular, access to an optimal solution of a big-$M$  formulation \color{black} does not mitigate the computational difficulty of verifying big-$M$ correctness.  This also limits the usefulness of iterative tuning procedures that increase the big-$M$ values and seek stability of the obtained optimal solutions. Such stability does not certify bilevel correctness, and \textit{independently justified} big-$M$ bounds remain necessary whenever such a guarantee is required.  \color{black} 

  \textbf{Funding.}
 This research did not receive any specific grant from funding agencies in the public, commercial, or not-for-profit sectors.

 
 
 \singlespacing
 \bibliographystyle{apa}
 \bibliography{bibliography}
\end{document}